\documentclass[12pt]{article}

\usepackage[english]{babel}
\usepackage[utf8]{inputenc}
\usepackage[T1]{fontenc}

\usepackage{graphicx}
\usepackage{caption, subcaption}
\usepackage[title]{appendix}
\usepackage{tikz}
\usepackage{pgfplots}
\usepackage{cutwin}
\graphicspath{ {images/} }
\usetikzlibrary{intersections,decorations.pathreplacing,decorations.markings,calc,angles,quotes,arrows.meta,pgfplots.fillbetween,patterns}

\usepackage[a4paper,top=2.9cm,bottom=2cm,left=2.5cm,right=2.5cm,marginparwidth=1.75cm]{geometry}

\usepackage{amsmath,yhmath}
\usepackage{amsthm}
\usepackage{bm}
\usepackage{amssymb}
\usepackage{hyperref}
\usepackage{cleveref}
\usepackage{epstopdf}
\usepackage{comment}
\usepackage{todonotes}
\usepackage{color}
\usepackage{float}

\usepackage[numbers]{natbib}

\newtheorem{thm}{Theorem}[section]

\newtheorem{prop}[thm]{Proposition}

\theoremstyle{definition}

\numberwithin{equation}{section}

\newcommand{\R}{\mathbb{R}}

\usepackage{soul}
\sethlcolor{cyan}
\usepackage{xcolor}

\makeatletter
\newcommand{\neutralize}[1]{\expandafter\let\csname c@#1\endcsname\count@}
\makeatother

\title{Finite Element Convergence Analysis For \\ Wave Equations With Time-Dependent Coefficients}

\author{
	Oussama Al Jarroudi\thanks{\footnotesize Department of Mathematics and Computer Science, 
		University of Basel, 
		Spiegelgasse 1, CH-4051 Basel, Switzerland (oussama.aljarroudi@unibas.ch, marcus.grote@unibas.ch).} \and Marcus J. Grote\footnotemark[1]}

\date{\today}

\begin{document}
	\maketitle
	
	\begin{abstract}
	Error estimates are proved for finite element approximations to the solution of second-order hyperbolic partial differential equations with coefficients varying in both space and time. Optimal rates of convergence in the energy norm are proved for the semi-discrete Galerkin finite element solution by introducing a time-dependent Ritz-like projection. Numerical experiments corroborate the rates of convergence and illustrate the localized wave field enhancement in a chain of time-modulated subwavelength resonators.
\end{abstract}
\vspace{0.5cm}
\noindent{\textbf{MSC codes}: 35J05, 35C20, 35P20
		
\vspace{0.2cm}
\vspace{0.5cm}	

\section{Introduction}
Wave phenomena are central to a wide range of scientific and engineering applications in acoustics, 
electromagnetics or elasticity. Traditionally, 
numerical methods, be they finite element or finite difference methods, have primarily addressed wave propagation in passive (or static) media with material properties varying in space but not in time.
Recent advances in material sciences, however, in particular the development
 of time-modulated metamaterials, have brought renewed attention to wave propagation in media whose properties vary in both space and time. Indeed,
these engineered materials enable fine-grained, dynamic control over parameters such as refractive index and impedance, unlocking novel wave-matter interactions
 with no analogue in the static case \cite{ammari2021time}. 
 From a mathematical perspective, these wave phenomena are modeled by second-order hyperbolic partial differential equations with coefficients that depend explicitly 
on both space and time, which introduces new analytical and computational challenges.

To this day, the Finite Element Method (FEM) remains probably 
the most versatile numerical method for the simulation of wave phenomena, as it
naturally allows for high-order polynomial approximations, a key feature to efficiently represent the oscillatory nature of waves, but also accomodates complex geometry and discontinuous material interfaces.
There is a long history of FEMs for second-order wave equations, starting in the 70's with the 
seminal work of Dupont \cite{dupont1973l2} and Baker \cite{baker1976,baker1976error} who first proved optimal $H^1$- and $L^2$ convergence rates for standard conforming FEM. Later, convergence was proved 
 in situations of low regularity \cite{larsson1991finite}, and more recently in polygons with corner singularities
 \cite{Schwab}.
For non-conforming discontinuous Galerkin (DG) FEMs, convergence for the second-order 
wave equation was established during the past twenty years (e.g., \cite{Riviere,GroteSchneebeliSchoetzauWAVE,GS09,Ern,Dassi}).

While the FE convergence theory for wave equations with spatially variable coefficients is well-established, 
it remains markedly less developed for coefficients that vary simultaneously in space and time, as 
the explicit time-dependence of the differential operator breaks 
the time-translational symmetry that underpins standard energy estimates.
%
In 1986, the seminal work of Bales \cite{bales1986higher} first established the convergence of the 
FEM for the wave equation when the coefficient in the elliptic operator depended on time. However, that analysis did not 
encompass the gain/loss term in modern time-modulated metamaterials due to the time derivative of the material coefficient, which periodically dampens or drives the energy of the system.

%
%

Our work bridges this gap as it presents a finite element error analysis for a general class of second-order hyperbolic partial differential equations whose material properties vary in both space and time. 
Here, the main challenge resides in controlling the error growth induced by the spatio-temporal variations of the medium's properties. To overcome this difficulty,
 we introduce a novel time-dependent Ritz-like projection, following ideas from \cite{wheeler1973priori, lin1991ritz}.  When combined with a duality argument reminiscent
of the Aubin-Nitsche trick, this projection enables the derivation of optimal a priori error estimates in the $H^1$-norm for the semi-discrete FEM solution. 
%
%
%
%
%
%
%
    While our primary focus here lies in the acoustic wave equation with space-time modulated bulk modulus,
 $\kappa(\mathbf{x},t)$, and mass density, $\rho(\mathbf{x},t)$, the
%
%
%
mathematical framework also applies to electromagnetics, for instance, when waves propagate through a
medium whose permittivity or permeability are actively modulated in time.
%

The remaining part of this work is structured as follows: In Section 2, we introduce the strong formulation of the model problem. Next, in Section 3, we state the semi-discrete Galerkin finite element formulation. In Section 4, we introduce a time-dependent Ritz-like projection, which will prove key in our analysis and lead to optimal $H^1$-error estimates in Section 5. Finally, in Section 6, we provide numerical experiments which
validate our theory and but also illustrate the intriguing physical phenomena exhibited by waves propagating across space-time modulated media.

\section{Wave equation in a time-varying medium}
Let $\Omega \subset \R^d$, $d\geq 1$, be an open, bounded Lipschitz-domain 
either convex or with a smooth boundary $\partial\Omega$, and 
let $J = (0, T)$ be a time interval up to a final time $T>0$. We consider the following time-dependent acoustic wave equation:
	\begin{equation}\label{eq:wave}
		\frac{\partial}{\partial t } \left(\frac{1}{\kappa(x,t)} \frac{\partial u(x,t)}{\partial t}\right) - \nabla \cdot \left( \frac{1}{\rho(x,t)} \nabla u(x,t) \right) = f(x,t), \qquad (x, t) \in \Omega \times J
	\end{equation}
	with initial conditions
		\begin{equation}\label{eq:ic}
		u(x,0) = u_0(x), \qquad 
		\frac{\partial }{\partial t }u(x,0)  = v_0(x), \qquad\qquad x \in \Omega
	\end{equation}
	and homogeneous Dirichlet boundary conditions for simplicity. 
	Both the bulk modulus, $\kappa(x,t)>0$, and the density, $\rho(x,t)>0$,
	may vary in space and time and are assumed sufficiently smooth;
	moreover, there exist two positive constants $C_{\star}$ and $C^{\star}$ such that
\begin{equation} 
0 < C_{\star} \leq 1/\kappa(x,t) , \, 1/\rho(x,t) \leq C^{\star}, \qquad\left( x,t \right) \in \Omega \times J.
\end{equation}

Rewriting \eqref{eq:wave} equivalently as
\begin{equation}
\label{eq:kapparho}
\frac{1}{\kappa(x,t)} \frac{\partial^2 u}{\partial t^2} -\frac{\kappa_t(x,t)}{\kappa^{2}(x,t)}  \frac{\partial u}{\partial t}  - \nabla \cdot\left(\frac{1}{\rho(x,t)}\nabla u\right)= f(x,t), \qquad\left( x,t \right) \in \Omega \times J,
\end{equation} 
  leads us to consider in our analysis the following more general form:
\begin{equation} \label{eq:wave_viscoelast}
\frac{1}{\kappa(x,t)} u_{tt} + \mathcal{B}(t) u_{t} + \mathcal{A}(t) u = f(x,t),
\end{equation} 
where $\mathcal{A}(t)$ is a self-adjoint, coercive and linear second-order elliptic partial differential 
%
and $\mathcal{B}(t)$ an arbitrary linear bounded  operator.
The well-posedness of the variational formulation of \eqref{eq:wave_viscoelast} follows from 
(\cite{lions1965some}, Theorem 1.1) under hypotheses about the temporal regularity of $\mathcal{A}(t)$ and $\mathcal{B}(t)$.
Although $1/\kappa(x,t)$ and $1/\rho(x,t)$ in \eqref{eq:kapparho} are both positive and bounded below away from zero, the time derivative $\kappa_t(x,t)$ of $\kappa(x,t)$
will generally change sign over time and thus either dampen or amplify the wave energy.
%
%

\section{Finite element formulation}
Let $H^{r}(\Omega)$ denote the standard Sobolev space on $\Omega$ with norm $\| \phi \|^2_{r} =  \sum_{|\alpha| \leq r} \| \partial^\alpha \phi \|_{L^2(\Omega)}^2 $ and
let $H_{0}^{1}(\Omega)$ be the closure of $C_{0}^{\infty}(\Omega)$ in $H^1(\Omega)$. The standard $L^2$ inner product and norm are denoted by $(\phi,\psi)$ and $\| \phi \| = (\phi, \phi)^{1/2}$, respectively. 

Given a regular partition of $\Omega$ into simplicial elements of mesh size $h>0$,
we seek the numerical solution in finite-dimensional subspaces \( S_h \subset H_{0}^{1}(\Omega)\) parameterized by $h$. We assume that $S_h$ has the following standard approximation property: there exists a constant $C > 0$, independent of $h$, such that 
\begin{equation} \label{eq:approx_prop}
\min_{\chi \in S_{h}} \left( \| v - \chi \| + h \| v - \chi \|_{1} \right) \leq C h^{r} \| v \|_{r} \quad \forall  v \in H_{0}^{1}(\Omega) \cap H^{r}(\Omega),
\end{equation}
for $r\geq 2$.
Standard ($H^1$-conforming) finite element spaces $S_h$ of continuous, piecewise polynomials of degree $r-1$, 
for instance, satisfy \eqref{eq:approx_prop} for a sequence of shape-regular and quasi-uniform triangulations. 


The semidiscrete finite element approximation to the solution $u$ of \eqref{eq:wave} is defined as a map $U_h : [0, T) \rightarrow S_{h}$ satisfying
\begin{align}\label{weakeq:wave}
\left( \frac{1}{\kappa(\cdot, t)} U_{h,tt}(t), \chi \right) + \left( a(\cdot, t) \nabla U_h(t), \nabla\chi \right) + \left( b(\cdot, t) U_{h,t}(t), \chi \right) = \left( f(t), \chi \right)
\end{align}
for all $\chi \in S_{h}$ and $t \in J$, where $a(x,t) = 1 / \rho(x,t)$ and $b(x,t) = -\kappa_t(x,t) / \kappa^{2}(x,t)$, while the subscript $t$ denotes partial differentiation with respect to time. 
Although $a(x,t)$ is positive and bounded below away from zero, $b(x,t)$ will generally change sign and thus either dampen or amplify the wave energy.
The initial conditions
\begin{equation}
\label{eq:icfem}
U_h(0)  = u_{0,h}\in S_{h}, \qquad U_{h,t}(0) = v_{0,h} \in S_{h}
\end{equation}
are chosen to satisfy
\begin{equation} \label{eq:initial_approx}
\| u_{0,h}\ - u_0 \|_1 \leq C h^{r-1} \| u_0 \|_{r},\qquad \|v_{0,h}  - v_0 \| \leq C h^r \| v_0 \|_{r}\,.
\end{equation}

If we let $A(t; \cdot, \cdot)$ and $B(t; \cdot, \cdot)$ denote the time-dependent bilinear forms on $H_{0}^{1}(\Omega) \times H_{0}^{1}(\Omega)$ associated with the  operators  $\mathcal{A}( t)$ and $\mathcal{B}( t)$, respectively, we can rewrite \eqref{weakeq:wave} equivalently as
\begin{align}
\label{eq:weakform}
&\left(\frac{1}{\kappa(\cdot, t)}U_{h,tt}(t), \chi\right) + B\left(t; U_{h,t}(t), \chi\right) + A(t; U_h(t), \chi)  = (f(t), \chi),
\end{align} 
for all $\chi \in S_h,  t > 0$.

 \section{A time-dependent projection}
The error analysis relies on a time-dependent Ritz-like projection $w(t):[0,T] \to S_h$, which maps the exact solution $u(t)$ onto the finite element space. It incorporates the gain/loss term $b(x,t)\,\partial_t u(t)$ from  \eqref{weakeq:wave} so that $w(t)$ faithfully reflects the attenuation or amplification in the true solution. Once we have shown that the time-dependent projection is well-defined, we derive upper bounds for the error between $w(t)$ and $u(t)$.

For $\gamma>0$ a positive and sufficiently large constant,
we define the {\em time-dependent projection} $w(t) : [0,T] \to S_h$ of the exact solution $u(t)$ as the solution of the first-order differential equation
\begin{equation}
\label{TimeDepProj}
(a(.,t) \nabla(w-u), \nabla \chi)+\gamma(\nabla(w-u)_{t}, \nabla \chi)+\left(b(.,t)\frac{\partial}{\partial t}(w-u), \chi\right)=0, \quad \forall\chi \in S_h, t > 0,
  \end{equation}
  with initial condition $w(0) \in S_h$ given by the standard Ritz projection of $u_0$, for instance:
  \begin{equation}
\label{eq:icw}
(\nabla(w(0)-u_0), \nabla \chi)=0, \quad \forall\chi \in S_h.
  \end{equation}
  
The term involving \(\gamma\) in equation \eqref{TimeDepProj} acts on the one hand as a penalization factor;
specifically, it counteracts negative values of $b(x,t)$, which actively amplify wave energy, and thus 
guarantees the stability and well-posedness of the time-dependent projection $w(t)$.
 On the other hand, it introduces the diffusion needed for estimating the error  \( \left\| w(t)-u(t) \right\|_{L^{2}(\Omega)} \) 
below through a modification of the standard  Aubin-Nitsche trick in conjunction with results from elliptic regularity theory.
%

To show that $w(t)$ is indeed well-defined, let
$ \left\{\psi_k\right\}_{k=1}^N,$ denote a basis of $S_h$. Then
\[
w(x, t) = \sum_{k=1}^N c_k(t) \psi_k(x), \qquad c_k(t)\in\R, 
\]
and we can rewrite \eqref{TimeDepProj} equivalently as 
\begin{equation}
\label{Proj-exisuniq}
\left ( \gamma \bm{K + B_\kappa}(t)\right) \frac{d}{d t} \underline{C}(t) + \bm{A_\rho}(t) \underline{C}(t) = \underline{F}(t), \qquad 0<t<T
\end{equation}
for the unknown coefficients $\underline{C}(t) = \big(c_1(t), \ldots, c_N(t)\big)^\top $ of $w(t)$. 
Here 
\[\bm{A_\rho}(t) = \big(a(.,t) \nabla \psi_k, \nabla \psi_\ell\big)_{1\leq k, \ell \leq N}\] 
denotes a time-dependent stiffness matrix weighted by  $a(x,t)= 1/\rho(x,t)$
\[\bm{K}= \big( \nabla \psi_k, \nabla \psi_\ell\big)_{1\leq k, \ell \leq N}\] 
the standard stifness-matrix, whereas
\[\bm{B_\kappa}(t) = \big(b(.,t) \psi_k, \psi_\ell\big)_{1\leq  k, \ell \leq N}\] 
denotes a time-dependent mass matrix weighted by $b(x,t)=-\kappa_t(x,t) / \kappa^{2}(x,t)$, while the right-hand side
$\underline{F}(t) = \big(F_1(t), \ldots, F_N(t)\big)^\top $ is given by
\[
F_\ell(t) = \big(a(.,t) \nabla u, \nabla \psi_\ell\big) +\gamma \big(\nabla u_t, \nabla \psi_\ell\big) + \big(b(.,t) u_t, \psi_\ell\big), \qquad \ell = 1, \ldots, N.
\]

Since $\bm{B_\kappa}(t)$ and $\bm K$ are symmetric, while $\bm K$ is also positive definite, 
the $N\times N$ system of linear ordinary differential equations \eqref{Proj-exisuniq} has a unique solution
$\underline{C}(t)$ if $\gamma \bm{K +B_{\kappa}}(t)$ remains positive definite for all $t>0$. 
Hence for $\gamma > \Lambda \,( = \max\{0,-\lambda_{\min}({\bm B_\kappa}(t))/\lambda_{\min}({\bm K})\})$
sufficiently large, there exists a unique solution $\underline{C}(t)$ of \eqref{Proj-exisuniq} and therefore 
$w(t)$ is well-defined for $t\in J=(0,T)$; note that $\Lambda$ is bounded above independently of $h$ 
for shape-regular and quasi-uniform meshes.

%
%
%
%
%

For the subsequent analysis, we remark that \eqref{TimeDepProj}  admits 
the following equivalent integral formulation:
\begin{align}
\label{Proj-integform}
\gamma &(\nabla(w-u)_{t}(t), \nabla \chi)+\left(b(.,t)(w-u)_{t}(t), \chi\right)+ \int_0^t\left(a(.,t) \nabla(w-u)_{t}(\tau), \nabla \chi\right)  d \tau   \\ 
&+(a(.,t) \nabla(w-u)(0), \nabla \chi)=0, \qquad \forall \chi \in S_h, \quad t\in J=(0,T). \notag
  \end{align}
Note that $a(.,t)$ under the integral sign is merely evaluated at the upper limit of integration.

Next, we consider the second-order elliptic differential operator $\tilde{B}(t)$, 
\begin{equation}
\label{Def-Btilde}
\tilde{B}(t)u= -\gamma \Delta u + b(x,t) u
\end{equation}
and denote by $\tilde{B}(t;.,.) $ the associated symmetric bilinear form; again, for $\gamma$ 
sufficiently large $\tilde{B}(t;.,.) $ is also coercive on $H^1_0(\Omega)$ uniformly in time.
Moreover, we denote by  $P_h=P_h(t)$ : $H^1_0(\Omega) \rightarrow S_h$ the (time-dependent) Ritz projection associated with the bilinear form $\tilde{B}(t;.,.)$: For every $ v \in H_0^1(\Omega)$,
\begin{align}
\label{defritz}
\tilde{B}\left(t ; P_h(t) v-v, \chi\right)=0, \qquad \forall \chi \in S_h, \quad t \in J.
\end{align}
Thus, equation \eqref{Proj-integform}, which defines the time-dependent projection $w(t)$ in integral form, 
can be expressed in terms of the bilinear forms  $A(t;.,.) $  and $\tilde{B}(t;.,.) $ as      
\begin{align}\label{Proj-bintegform}
\tilde{B} \left(t ;(w-u)_t, \chi\right)+\int_0^t A\left(t ;(w-u)_t(\tau), \chi\right) d \tau +A(t; (w-u)(0),\chi) =0, 
\end{align}
for all $\chi \in S_{h}$ and $t\in J$.
 
We shall now derive upper bounds for $\eta=w-u$. To do so, we let $\eta_x = \nabla_x \eta$ denote an arbitrary (scalar) component of $\nabla \eta$ and recall that
\[ \|\eta_x \|=\sup \left\{\left(\eta_x, \varphi\right) | \varphi \in C_0^{\infty}\left(\Omega\right), \|\varphi \|  =1\right\}. \]
%
 For any such $\varphi$ and (fixed) time $0<t<T$, let $\psi = \psi(.,t)\in H_0^{1}(\Omega)$ be the solution of the elliptic boundary-value problem
\begin{equation}
\label{bvppsi}
\left\{\begin{array}{rll}
\tilde{B}(t) \psi & =-\nabla_x \varphi & \quad\mbox{ in }\Omega \\
\psi & =0 & \quad\mbox{ on }\partial\Omega 
\end{array}\right.
\end{equation}
with $\tilde{B}(t)$ defined in \eqref{Def-Btilde}.
We multiply \eqref{bvppsi} by $\psi$, integrate by parts in (the single coordinate) $x$ and use that $\|\varphi\|=1$ to obtain
 \[
   \beta \|\psi\|_{1}^2 \leq \tilde{B}(t; \psi, \psi) = -(\nabla_{x}\varphi, \psi) = (\varphi, \nabla_{x}\psi) \leq \|\varphi\| \|\nabla\psi\| \leq \|\psi\|_{1}
   \]
   where $\beta>0$ is the coercivity constant of \(\tilde{B}(t;\cdot,\cdot)\). Hence, we conclude that
\begin{align} \label{ellipreg}
\|\psi(t)\|_{1} \leq \frac{1}{\beta}, \qquad \forall t\in J.
\end{align}
%
%
%
%
%
Similarly, multiplying \eqref{bvppsi} by $\eta_t$ and integrating by parts in $x$ also yields
\begin{align}\label{AubNietsch}
\left(\nabla_x \eta_t, \varphi\right) =-\left(\eta_t ,\nabla_x \varphi\right)=\left(\eta_t, \tilde{B}(t) \psi\right)
= \tilde{B}\left(t ; \eta_t , \psi\right).
\end{align}      
   
On the other hand, setting $\chi = P_h(t) \psi(t) \in S_h$ in \eqref{Proj-bintegform} implies that
\begin{align*}
\tilde{B}\left(t; \eta_t, \psi-P_h(t) \psi\right)&=\tilde{B}\left(t; \eta_t ,\psi\right)-\tilde{B}\left(t ; \eta_t, P_h(t) \psi\right)\\
&=\tilde{B}\left(t; \eta_t ,\psi\right)+\int_0^t A\left(t ;\eta_t(\tau), P_h(t) \psi \right) d \tau + A\left(t ;\eta(0), P_h(t) \psi \right)
\end{align*}
Therefore, we infer that                                
\begin{align*}
\tilde{B}\left(t; \eta_t ,\psi\right)=\tilde{B}\left(t; \eta_t, \psi-P_h(t) \psi\right)-\int_0^t A\left(t ;\eta_t(\tau), P_h(t) \psi \right) d \tau - A\left(t ;\eta(0), P_h(t) \psi \right)
\end{align*}                
which together with \eqref{AubNietsch} can be rewritten as
\begin{align}\label{estGradetat}
\left(\nabla_x \eta_t, \varphi\right)=\tilde{B}\left(t; \eta_t, \psi-P_h(t) \psi\right)-\int_0^t A\left(t ;\eta_t(\tau), P_h(t) \psi \right) d \tau - A\left(t ;\eta(0), P_h(t) \psi \right).
\end{align}    

To estimate the gradient of $\eta_t$ in the $L^2\text{-Norm}$, we shall bound the three terms on the right of 
\eqref{estGradetat} separately. For the first term, we have
\begin{align*}
\tilde{B}\left(t; \eta_t, \psi-P_h(t) \psi\right)&= \tilde{B}\left(t; w_t-u_t, \psi-P_h(t) \psi\right)\\
&=\underbrace{\tilde{B}\left(t; w_t, \psi - P_h(t) \psi\right)}_{= 0}-\tilde{B}\left(t; u_t, \psi-P_h(t) \psi\right)\\
&=\tilde{B}\left(t; u_t, (P_h(t) -I)\psi\right)\\
&=\tilde{B}\left(t; (P_h(t) -I) u_t, \psi\right)\\
&=\int_{\Omega} \tilde{B}(t) \psi\left(P_h(t)-I\right) u_t(x) \,dx
\end{align*}   
where whe have used definition \eqref{defritz} twice for $P_h$ either applied to $u_t$ or $\psi$.
                     
Since $\psi \in H_0^1 $ solves \eqref{bvppsi},
we obtain through integration by parts 
\begin{align*}
 \int_{\Omega} \tilde{B}(t) \psi\left(P_h(t)-I\right) u_t(x) d x&=\int_{\Omega}-\nabla_x \varphi\left(P_h(t)-I\right) u_t(x) d x\\
&=\int_{\Omega} \varphi\nabla_x \left(P_h(t) u_t-u_t\right) \, dx.
\end{align*}
This yields the following estimate for the first term on the right side of \eqref{estGradetat}:
 \begin{align*}
| \tilde{B}\left(t, \eta_t, \psi-P_h(t) \psi\right)| &=| \left(\nabla_x\left(P_h(t) u_t - u_t\right), \varphi\right)| \\
 &\leqslant C h^{r-1}\left\|u_t\right\|_{r},
\end{align*}
where we have used the standard error estimate for the Ritz-projection,
\begin{align*}
\|\nabla(P_h(t) u-u)\| \leqslant C h^{r-1}\|u\|_r \,.
\end{align*}
%

Next, we estimate the second term on the right side of \eqref{estGradetat} 
\begin{align*}
\left| \int_0^t A\left(t; \eta_t(\tau), P_h(t) \psi\right) d \tau \right| & \leqslant C \int_0^t\left\|\nabla \eta_t(\tau)\right\|\left\|\nabla P_h(t) \psi\right\| d \tau \\
& \leqslant C \int_0^t\left\|\eta_t(\tau)\right\|_{1}\|\psi\|_{1} d \tau
\end{align*}
where we have first applied the Cauchy-Schwarz inequality and then used 
\begin{align*}
 \left\|P_h u-u\right\|_1 \leqslant C\|u\|_1 \,.
\end{align*}
From \eqref{ellipreg}, we thus conclude that 
\begin{align*}
\left|\int_0^t A\left(t; \eta_t(\tau), P_h(t) \psi\right) d \tau \right| \leqslant C \int_0^t\left\|\eta_t(\tau)\right\|_{1} d \tau\,.
\end{align*}
\vspace{0.5em}

Similarly for the third term on the right side of \eqref{estGradetat}, we derive the upper bound
\begin{align}
\left|A\left(t ;\eta(0), P_h(t) \psi \right) \right| \leqslant C \left\| \eta(0) \right\|_1\,.
\end{align}

Finally, we combine the above three estimates with \eqref{estGradetat} to obtain
\begin{align}
\left|\left(\nabla_x\eta_t, \varphi\right)\right|
& \leqslant\left|\tilde{B}\left(t; \eta_t, \psi-P_h(t) \psi\right)\right|+\int_0^t\left|A\left(t; \eta_t(\tau), P_h(t) \psi\right)\right| d \tau \\
& \ +\left|A\left(t; \eta(0), P_h(t) \psi\right)\right| \\
& \leqslant C h^{r-1}\left\|u_t\right\|_r+ C \int_0^t\left\|\eta_t(\tau)\right\|_{1} d \tau + C \left\|\eta(0)\right\|_{1} \,.
\end{align}
Since $\left\|\nabla_x \eta_t\right\|=\sup \left\{\left(\nabla_x \eta_t , \varphi\right), \varphi \in C_0^{\infty}(\Omega),\| \varphi\|=1\right\},$                    
this leads to the upper bound
\begin{align*}
\left\|\eta_t\right\|_1 \leqslant  C \left\|\eta(0)\right\|_{1}  + C h^{r-1}\left\|u_t\right\|_r+C \int_0^t\left\|\eta_t(\tau)\right\|_1 d \tau .
\end{align*}

From \eqref{eq:icw} we infer that $\left\|w(0) - u_0\right\|_{1} = \left\|\eta(0)\right\|_{1} \leq C h^{r-1 }$. Thus, we conclude by Gronwall's lemma that 
%
\begin{align*}
\left\|\eta_t\right\|_{1} \leqslant C h^{r-1}\left\{\left\|u_t\right\|_{r}+\int_0^t\left\|u_t(\tau)\right\|_r d \tau\right\},
\end{align*}
or equivalently
\begin{align}
 \label{etatH1estim}
\left\|\eta_t\right\|_{1} \leq C h^{r-1}\left\|u_t\right\|_{0, r} \qquad \forall t\in \bar{J},
\end{align}
where we have introduced the notation
\begin{equation*}
\left\|u(t)\right\|_{k, r}=\sum_{j=0}^k\left\{\left\|\partial_{t}^j u(t)\right\|_{r}+\int_0^t\left\|\partial_{t}^j u(\tau)\right\|_{r} d \tau\right\}.
\end{equation*}

Given the upper bound \eqref{etatH1estim} for $\eta_t$, we shall estimate  $\eta$ by using
\begin {align} 
\label{etatIneq}
\left\|\eta(t)\right\|_{1}^{2} \leqslant C \left\| \eta(0)\right\|_{1}^{2} + C \int_0^t \left\|\eta_t(\tau)\right\|_{1}^{2} d \tau.
\end{align}
Indeed since
 \begin{align*} 
\eta(t) = \eta(0) + \int_0^t  \eta_t(\tau) d\tau,
\end{align*}
we easily infer using Young's inequality for any $\epsilon>0$ that 
\begin{align*}
\left\|\eta(t)\right\|^{2}  &= \left\| \eta(0)\right\|^{2}  +2(\eta(0),\int_0^t \eta_t(\tau) d\tau) +\left\| \int_0^t \eta_t(\tau) d\tau \right\|^{2} \\
   &\leqslant \left\| \eta(0)\right\|^{2}  + 2 \epsilon \left\| \eta(0)\right\|^{2} + \frac{2}{4\epsilon} \left\| \int_0^t \eta_t(\tau) d\tau \right\|^{2} + \left\| \int_0^t \eta_t(\tau) d\tau \right\|^{2} \\
   & \leqslant (1+2\epsilon) \left\| \eta(0)\right\|^{2}  + \left(1+\frac{1}{2\epsilon}\right) \left (  \int_0^t  \left\|  \eta_t(\tau) \right\| d\tau  \right)^2\\
   &\leqslant  (1+2\epsilon) \left\| \eta(0)\right\|^{2}  + \left(1+\frac{1}{2\epsilon}\right) \,T    \int_0^t  \left\|  \eta_t(\tau) \right\|^2 d\tau .
\end{align*}
The same argument applied to $\nabla \eta(t)$ then yields \eqref{etatIneq} after summation.

Finally, we obtain by \eqref{etatH1estim}
\begin{align} \label{etaH1estim}
\left\|\eta(t) \right\|_1 \leq C h^{r-1}\left\|u_t\right\|_{0, r} \quad \text{for t} \quad \in \bar{J}
\end{align}

We shall now estimate $\eta_t$ with respect to the $L^2$-norm by a modification of the standard duality argument.
Again, let $\psi$ be the solution to the elliptic boundary value problem
\[
\begin{cases}\tilde{B}(t) \psi=\varphi & \text { in } \Omega \\ \psi=0 & \text { on } \partial \Omega\end{cases},
\]
for $\varphi  \in C_0^{\infty}\left(\Omega\right)$ smooth and  $\|\varphi\|=1$. From elliptic regularity theory (eg. \cite{schechter1963}) we know that $\|\psi\|_{2} \leqslant C_2\|\varphi\|=C_2 $. 

As previously we have for $\eta=w-u$,
\begin{align*}
& \left(\eta_{t,} \varphi\right)=\int_{\Omega} \eta_t(x, t) \varphi(x) d x=\int_{\Omega} \eta_t(x,t) \tilde{B}(t) \psi(x) d x \\
& =\tilde{B}\left(t ; \eta_t , \psi\right) \\
& =\tilde{B}\left(t ; \eta_t, \psi-P_h(t) \psi\right)+\tilde{B}\left(t ; \eta_t, P_h (t)\psi\right). \\
& =\tilde{B}\left(t ; \eta_t, \psi-P_h(t) \psi\right)-\int_0^t A\left(t; \eta_t(\tau), P_h(t) \psi\right) d \tau - A\left(t; \eta(0), P_h(t) \psi\right)  \\
&=\tilde{B}\left(t ; \eta_t, \psi-P_h(t) \psi\right)+\int_0^t A\left(t; \eta_t(\tau), \psi - P_h(t) \psi\right) d \tau -\int_0^t A\left(t; \eta_t(\tau), \psi\right) d \tau\\ 
&- A\left(t; \eta(0), P_h(t) \psi\right)  .
\end{align*}
Hence, the triangle inequality immediately yields
\begin{align}
\left|\left( \eta_t, \varphi \right)\right| 
& \leq \overset{(\mathrm{I})}{\left|\tilde{B}\left(t; \eta_t , \psi - P_h(t) \psi\right)\right|}
+ \overset{(\mathrm{II})}{\int_0^t \left|A\left(t; \eta_t (\tau), \psi - P_h(t) \psi\right)\right| d \tau} \notag \\
& \quad + \underset{(\mathrm{III})}{\int_0^t \left|A\left(t; \eta_t(\tau), \psi\right)\right| d \tau} +  \underset{(\mathrm{IV})}{\left|A\left(t; \eta(0), P_h(t) \psi\right)\right| }.
\end{align}
To derive an upper bound for $\eta_t$, we shall estimate each individual term (I)--(IV) separately.

For (I), we have following similar arguments as previously:
\begin{align*}
\tilde{B}\left(t ; \eta_t, \psi-P_h(t) \psi\right)&=\tilde{B}\left(t ; w_t-u_t,\left(I-P_h(t)\right) \psi\right)\\
& =\underbrace{\tilde{B}\left(t; w_t, \left(I-P_h(t)\right) \psi\right)}_{= 0}- \tilde{B}\left(t; u_t, \left(I-P_h(t)\right) \psi\right)\\
& = \tilde{B}\left(t; \left(P_h(t)-I\right) u_t, \psi\right) \\
& =\left(\left(P_h(t)-I\right) u_t, \varphi\right),
\end{align*}
where we have used the definitions of $\eta$, $\psi$ and $P_h(t)$. Since $\|\varphi\|=1$, we thus immediately conclude that
\begin{align}
\left|\tilde{B}\left(t; \eta_t , \psi - P_h(t) \psi\right)\right| = \left|\left(\left(P_h(t)-I\right) u_t, \varphi\right) \right|  \leqslant C h^r\left\|u_t\right\|_{r} .
\end{align}

Next, for (III) we first integrate parts in the definition of $A\left(t; \eta_t(\tau), \psi(t) \right)$, which yields
\begin{align*}
\left|A\left(t; \eta_t(\tau), \psi(t) \right) \right|= \left|\int_{\Omega} a(x, t) \nabla \eta_t(x, \tau) \cdot \nabla \psi(x) d x \right|& \leqslant C\int_{\Omega}\left|\eta_t(x,\tau)\right|\left|\Delta \psi(x)\right| \,d x.
\end{align*}
By elliptic regularity, we infer that
\begin{align*}
\int_{\Omega}\left|\eta_t(x, \tau)\right|\left|\Delta \psi(x,t)\right| d x & \leqslant\left\|\eta_t(\tau)\right\|\|\psi(t)\|_{2}  \leqslant C_2\left\| \eta_t(\tau)\right\|, \label{RitzProj-ineq}
\end{align*}
which after time integration leads to the upper bound
\begin{align}
\int_0^t  \left|A\left(t; \eta_t(\tau), \psi \right)\right| d\tau  \leqslant C_2 \int_0^t\left\| \eta_t(\tau)\right\| d \tau .
\end{align}

For (II), we immediately infer by Cauchy-Schwarz that
\begin{align*}
\int_0^t \left| A\left(t; \eta_t(\tau), \psi- P_h(t) \psi\right)\right| d\tau  \leqslant C_2 \int_0^t\left\| \nabla \eta_t(\tau)\right\| \left\| \psi-P_h(t) \psi \right\| d \tau.
\end{align*}
Then, \eqref{etaH1estim} together with 
\begin{align} \label{H1ProjEst}
\left\| P_h(t) \psi - \psi \right\|_{1}  \leqslant  C_2 h \left\| \psi(t) \right\|_{2}  \leqslant C_2 h               
\end{align}       
 implies that 
\begin{align}
\int_0^t \left| A\left(t; \eta_t(\tau), \psi- P_h(t) \psi\right)\right| d\tau  \leqslant  C h^{r} \int_0^t\left\|u_t(\tau)\right\|_{0, r}\,d \tau .
\end{align}                     

Finally for term (IV), \eqref{H1ProjEst} immediately yields
\begin{align*}
\left|A\left(t; \eta(0), P_h(t) \psi\right)\right|  \leq C h \left\| \eta(0) \right\|_1 \leq C h^{r}.
\end{align*}

By combining the above four estimates for (I)--(IV), we obtain the upper bound
\begin{align*}
 \left|\left( \eta_t, \varphi \right)\right| \leqslant Ch^r +C h^r\left\|u_t\right\|_{r}  + C_2 \int_0^t\left\| \eta_t(\tau)\right\| d \tau + C h^{r} \int_0^t\left\|u_t(\tau)\right\|_{0, r}  d \tau.
 \end{align*}        
Taking the supremum over $\varphi$ with $\| \varphi\|=1$ yields the estimate
 \begin{align*}
 \left\| \eta_t(t) \right\| \leqslant Ch^r + C h^r\left\|u_t\right\|_{r,}  + C_2 \int_0^t\left\| \eta_t(\tau)\right\| d \tau + C h^{r} \int_0^t\left\|u_t(\tau)\right\|_{0, r}  \,d \tau.
 \end{align*}  
 
 To conclude we use Gronwall's lemma, which implies that 
 \begin{align}\label{etatL2estim}
 \left\| \eta_t(t) \right\| \leqslant C h^r\left\|u_t\right\|_{0, r}.  
  \end{align}  
  
Again by using  
\begin {align*}
\left\|\eta(t)\right\|^{2} \leqslant C \left\| \eta(0)\right\|^{2} + C \int_0^t \left\|\eta_t(\tau)\right\|^{2} d \tau
\end{align*}
we obtain by \eqref{etatL2estim}
\begin{align} 
\label{etaL2estim}
 \left\| \eta(t) \right\| \leqslant C h^r\left\|u_t\right\|_{0, r}.
\end{align}

In summary, we have proved the following theorem.
\begin{thm}
\label{theo1}
Let $u$ be the solution to \eqref{eq:wave} with $u_t \in L^{\infty}\left(0,T; H^r(\Omega)\right)$ for $r\geq 2$,
and $w(t)$ be the time-dependent projection defined in \eqref{TimeDepProj}.
Then, there exists a constant $C>0$, independent of $h$, such that it holds for all t $\leqslant T$:
\begin{align}
\left\| w(t)-u(t) \right\|_{L^2(\Omega)} + h \left\| w(t)-u(t)\right\|_{H^1(\Omega)} \leqslant  C h^r\left\|u_t(t)\right\|_{0, r} . 
\end{align}
\end{thm}
 
In the above proof of the upper bounds for $\eta(t) = \omega(t)-u(t)$ in the $L^2$- and $H^1$-norms, we have also derived in \eqref{etatH1estim} and \eqref{etatL2estim}
estimates for $\eta_t(t)$ in both norms. We shall now first derive  estimates for \(\eta_{tt}(t)\) by following the same procedure as for \(\eta_t(t)\). By induction this leads to 
the following more general result which provides upper bounds for the $k$-th time derivative of $\eta(t)$ in both the $L^2$- and $H^1$-norms.
\begin{prop}
\label{prop1}
Let $u$ be the solution to \eqref{eq:wave} with $ \partial_t^k u \in L^{\infty}\left(0,T; H^r(\Omega)\right)$ for $k \geq 1 $, $r\geq 2$,
and $w(t)$ be the time-dependent projection defined in \eqref{TimeDepProj}.
Then, there exists a constant $C_k>0$, independent of $h$, such that it holds for all t $\leqslant T$:
%
\begin{align}
\left\|\partial_t^k \left(w-u\right)(t) \right\|_{L^2(\Omega)} + h \left\| \partial_t^k \left(w-u\right)(t)\right\|_{H^1(\Omega)} \leqslant  C_k \,h^r\left\|u_t(t)\right\|_{k-1, r} . 
\end{align}
\end{prop}   
                    
 \begin{proof}
The case $k=1$ has already been established in our proof of Theorem \ref{theo1}. \\
For $k= 2$, we differentiate \eqref{Proj-bintegform} with respect to time:
\begin{align*}
\frac{\partial}{\partial t }\left[ \tilde{B} \left(t ;\eta_t, \chi\right)+\int_0^t A\left(t ;\eta_t(\tau), \chi\right) d \tau + A\left(t ;\eta(0), \chi\right) \right]=0, \quad \forall \chi \in S_{h}, \quad t\in J.
\end{align*}
The time derivetive of the first term, 
\begin{align*}
\frac{\partial}{\partial t } \tilde{B} \left(t ;\eta_t, \chi\right)= \tilde{B} \left(t ;\eta_{tt}, \chi\right) + \tilde{B}_{t} \left(t ;\eta_t, \chi\right),
\end{align*}
 where $ \tilde{B}_{t} \left(t ;\eta_t, \chi\right)$ denotes the bilinear form associated with the operator $\partial_t b(x,t) I$.
Similarly, the time derivative of the second term is
\begin{align*}
\frac{\partial}{\partial t } \int_0^t A\left(t ;\eta_t(\tau), \chi\right) d \tau =  A\left(t ;\eta_t(t), \chi\right) +  A_{t}\left(t ;\eta_t(t), \chi\right).
\end{align*}
Therefore, we obtain for  all $\chi \in S_{h}$ and $0<t<T$:
\begin{align}\label{etattProjder}
 \tilde{B} \left(t ;\eta_{tt}, \chi\right) + \tilde{B}_{t} \left(t ;\eta_t, \chi\right) + A\left(t ;\eta_t(t), \chi\right) +  A_{t}\left(t ;\eta_t(t), \chi\right)+ A_{t}\left(t ;\eta(0), \chi\right) =0. 
\end{align}
Following along the lines of our previous argument, we have for $\nabla_x \eta_{tt}$ an arbitrary component of $\nabla \eta_{tt}$:
\begin{align*}
\left(\nabla_x \eta_{tt}, \varphi\right)& = \tilde{B}\left(t ; \eta_{tt}, \psi\right) \\
&=\tilde{B}\left(t ; \eta_{tt}, \psi -P_h \psi \right) +\tilde{B}\left(t ; \eta_{tt}, P_h \psi \right) \\
&=\tilde{B}\left(t ; \eta_{tt}, \psi - P_h \psi \right) - \tilde{B}_{t} \left(t ;\eta_t, P_h \psi \right)\\
&- A\left(t ;\eta_t(t), P_h \psi\right)- A_{t}\left(t ;\eta_t(t),P_h \psi \right) - A_{t}\left(t ;\eta(0), P_h \psi  \right)  \\
&=\left(\nabla_x\left(P_h u_{tt}- u_{tt}\right), \varphi\right)  - \tilde{B}_{t} \left(t ;\eta_t, P_h \psi \right)\\
&- A\left(t ;\eta_t(t), P_h \psi\right)- A_{t}\left(t ;\eta_t(t),P_h \psi \right) - A_{t}\left(t ;\eta(0), P_h \psi  \right),
\end{align*}         
 where we have used \eqref{etattProjder} for the third equality. 
Again, the four terms on the right-hand side can be bounded as previously.  \\
For $k\geqslant 3$,  the proof follows by induction by repeatedly differentiating \eqref{etattProjder}.
 \end{proof}

 \section{Semi-discrete finite-element approximation}  
 Owing to the the time-dependent Ritz-like projection $w(t)$ introduced in the previous section, we shall now
 prove the following optimal $H^1$-error estimate for the semidiscrete Galerkin finite-element approximation of \eqref{eq:wave}.

\begin{thm} \label{thm:demidisc} 
Let $u$ be the solution to \eqref{eq:wave} with $u_t, u_{tt} \in L^{\infty}\left(0,T; H^r(\Omega)\right)$ for $r\geq 2$,
and  $U_h(t) \in S_h$ be the semi-discrete Galerkin FE-solution of \eqref{eq:weakform}.
Then, there exists a constant $C>0$, independent of $h$, such that it holds for all t $\leqslant T$:
%
%
\[
\|u - U\|_{L^{\infty}\left(0,T; H^1(\Omega)\right)} \leqslant C h^{r-1}.
\]
\end{thm} 
 \begin{proof}
 First, we split the error as 
  \[u(t)-U_h(t)= u(t)-w(t)+ w(t)-U_h(t)= -\eta(t) + \theta(t),\]
 where $w(t) \in S_h$ is the time-dependent projection defined in \eqref{TimeDepProj}.
 
 Then, from  \eqref{TimeDepProj} we easily infer that for every $\chi\in S_h$
 \begin{align*}
 & \left( \frac{1}{\kappa(t)}w_{tt}, \chi\right) + \left(a(t) \nabla w, \nabla \chi \right) + (b(t) w_t, \chi )= \left( \frac{1}{\kappa(t)}w_{tt}, \chi \right)+\left(a(t) \nabla u, \nabla \chi\right)\\
 &+\left(b(t) u_t, \chi \right) -  \gamma \left( \nabla w_t, \nabla \chi \right)+\gamma \left(\nabla u_t, \nabla \chi \right) \\
 & =( \frac{1}{\kappa(t)} w_{tt}, \chi ) - (  \frac{1}{\kappa(t)} u_{tt}, \chi )+( \frac{1}{\kappa(t)} u_{tt}, \chi)+\left(a(t) \nabla u(t), \nabla \chi \right)+(b(t) u_t, \chi ) \\
 &-\gamma \left(\nabla w_t, \nabla \chi \right)+ \gamma \left(\nabla u_{t}, \nabla \chi \right).
  \end{align*}
Here for the sake of conciseness, we omit the $x$ dependence in all coefficients, i.e., we write
$\kappa(t)$ instead of $\kappa(.,t)$, etc.
Since the exact solution $u(t)$ also satisfies \eqref{weakeq:wave}, $w(t)$ satisfies
 \begin{align}\label{eq:werreq}
\left( \frac{1}{\kappa(t)}w_{tt}(t), \chi\right) + \left(a(t) \nabla w(t), \nabla \chi \right) + (b(t) w_t(t), \chi )&= \left( \frac{1}{\kappa(t)} \eta_{tt}, \chi \right)+\left(f(t), \chi \right)\\
&-\gamma \left(\nabla \eta_t(t), \nabla \chi \right). \notag
  \end{align}
  By subtracting \eqref{weakeq:wave} from \eqref{eq:werreq} we obtain 
 the following error equation for $\theta = w - U_h$:
   \begin{equation}
   \label{eq:theta}	
 \left( \frac{1}{\kappa(t)} \theta_{tt}(t), \chi\right) + \left(a(t) \nabla \theta(t), \nabla \chi \right) + (b(t) \theta_t(t), \chi ) =
  \left( \frac{1}{\kappa(t)} \eta_{tt}(t), \chi \right)-\gamma \left(\nabla \eta_t(t), \nabla \chi \right)
  \end{equation}  
for all $\chi \in S_h$ and $t\in J=(0,T)$.     
               
Next, we let $\chi = \theta_t$ in \eqref{eq:theta}. Since we have
   \begin{align*}
    \left( \frac{1}{\kappa(t)}\theta_{tt}, \theta_t \right) =\frac{1}{2} \frac{d}{d t} \left\| \frac{1}{\sqrt{\kappa(t)}} \theta_t \right\|^2 +\frac{1}{2}  \left\| \frac{1}{ \kappa(t)} \theta_t \right\|^2,
   \end{align*}
%
and similarly
\begin{align*}
\left(a(t) \nabla \theta, \nabla \theta_t \right)=\frac{1}{2} \frac{d}{d t}(a(t) \nabla \theta, \nabla \theta)-\frac{1}{2}\left(a^{\prime}(t) \nabla \theta, \nabla \theta\right)
\end{align*}
and 
\begin{align*}
\left(\nabla \eta_t , \nabla \theta_t\right)= \frac{d}{d t }\left(\nabla \eta_t , \nabla \theta\right)-\left(\nabla \eta_{tt} , \nabla \theta\right),
\end{align*}
we can rewrite the error equation~\eqref{eq:theta} as
\begin{align*}
& \frac{1}{2} \frac{d}{d t}\left( \left\| \frac{1}{\sqrt{\kappa(t)}}\theta_t(t) \right\|^2 +(a(t) \nabla \theta , \nabla \theta)\right)  +   \frac{1}{2} \left\| \frac{1}{\kappa(t)}\theta_t(t) \right\|^2 +  \left(b(t) \theta_t, \theta_t\right)  = \left(\frac{1}{\kappa(t)} \eta_{t t}, \theta_t \right)\\
& +\gamma\left(\nabla \eta_{t t},\nabla \theta\right) -\gamma \frac{d}{d t}\left(\nabla \eta_t, \nabla \theta\right)+\left(a_t(t) \nabla \theta, \nabla \theta\right). \notag
\end{align*}

Next, we omit the third (positive) term on the left side and
integrate in time to obtain
\begin{align*}
& \frac{C_{*}}{2}\left\| \theta_t \right\|^{2}+\frac{1}{2}(a(t)\nabla \theta, \nabla \theta)-\frac{C^*}{2} \left\| \theta_t(0) \right\| ^2-\frac{1}{2}(a(0) \nabla \theta(0), \nabla \theta(0)) \leqslant \\
&\int_0^t\left(  \frac{1}{\kappa(\tau)} \eta_{t t}(\tau), \theta_t(\tau)\right) d \tau+\gamma \int_0^t\left(\nabla \eta_{tt}(\tau), \nabla \theta(\tau)\right) d \tau- \\
& \gamma\left(\nabla \eta_{t}, \nabla \theta\right)+\gamma\left(\nabla \eta_t(0), \nabla \theta(0)\right)+\int_0^t\left(a_t(\tau) \nabla \theta(\tau), \nabla \theta(\tau) \right) d \tau  - \int_0^t \left(b(\tau) \theta_t(\tau), \theta_t(\tau)\right) d \tau .
\end{align*}

By using Young's inequality we estimate the terms on the right side for $\varepsilon>0$ as
\begin{eqnarray*}
 \left|(\eta_{tt}, \theta_t)\right| &\leqslant& \varepsilon\left\|\eta_{t t}\right\|^2+\frac{1}{4 \varepsilon}\left\|\theta_t\right\|^2 \\
 \left|\gamma\left(\nabla \eta_{tt} , \nabla \theta\right)\right| &\leqslant& \gamma \varepsilon\left\|\nabla \eta_{tt}\right\|^2+\frac{1}{4 \varepsilon} \gamma\left\|\nabla \theta\right\|^2 \\
 \left|\left(a_t(t) \nabla \theta, \nabla \theta\right)\right| &\leqslant& C_{a }^{\prime}\|\nabla \theta\|^2 \\
 \left|\left(b_t(t) \theta_t, \theta_t \right)\right| &\leqslant& C_{b}\|\theta_t\|^2.
\end{eqnarray*}
This leads to
\begin{align*}
&\frac{C^*}{2}\|\theta_t\|^2+\frac{C_*}{2} \|\nabla \theta\|^2 \leqslant \frac{C^*}{2}\left\|\theta_t(0)\right\|^2+\frac{C^*}{2} \|\nabla \theta(0)\|^2+\gamma \varepsilon\left\|\nabla \eta_t(0)\right\|^2 \\
&+\frac{\gamma}{4 \varepsilon}\|\nabla \theta(0)\|^2+\varepsilon C^* \int_0^t\|\eta_{t t}(\tau)\|^2 d \tau+\frac{C^* +C_b}{4 \varepsilon} \int_0^t\left\|\theta_t(\tau)\right\|^2 d \tau \\
&+\varepsilon \gamma \int_0^t\|\nabla \eta_{t t}(\tau)\|^2 d \tau+\frac{\gamma}{4 \varepsilon} \int_0^t\|\nabla \theta(\tau)\|^2 d \tau + C_{a }^{\prime} \int_0^t\|\nabla \theta(\tau)\|^2 d \tau \\
&+ \frac{\gamma}{4 \varepsilon} \|\nabla \theta\|^2+\gamma \varepsilon\left\|\nabla \eta_t(t)\right\|^2 .
\end{align*}

Rearranging terms and choosing $ \epsilon > \gamma/C_*$, 
we thus obtain
\begin{align}
\label{eqtheta}
 \frac{C^*}{2}\left\|\theta_t\right\|^2+\frac{C_*}{4}\|\nabla \theta\|^2 \leqslant &\, C\mathcal{E}_0^2+
O\left(h^{2r-2}\right)+\frac{C}{4\varepsilon} \int_0^t\left\|\theta_t(\tau)\right\|^2 d \tau \notag \\
& +\left(C_a^{\prime}+\frac{\gamma}{4 \varepsilon}\right) \int_0^t\|\nabla \theta(\tau)\|^2 d \tau, 
\end{align}
where $C$ denotes a generic constant with $\mathcal{E}_0$ defined as
\begin{equation} 
\label{defeps0}
C \mathcal{E}_0^2= \frac{C^*}{2}\left\|\theta_t(0)\right\|^2+\left(\frac{ C^*}{2}+\frac{\gamma}{4 \varepsilon} \right) \|\nabla \theta(0)\|^2+\gamma \varepsilon\left\|\nabla \eta_t(0)\right\|^2 ,
 \end{equation}
  and we have used
   \begin{align*}
    \varepsilon C^{*} \int_0^t\left\|\eta_{tt}(\tau)\right\|^2 d \tau+\varepsilon \gamma \int_0^t\left\|\nabla \eta_{t t}(\tau)\right\|^2 d \tau + \gamma \varepsilon\left\|\nabla \eta_t(t)\right\|^2 =O\left(h^{2r-2}\right)
  \end{align*}
  because of Theorem \ref{theo1} and Proposition \ref{prop1} for $k=1,2$. Hence by Gronwall's lemma applied to
  \eqref{eqtheta}, we obtain
\begin{align}
\label{defeps1}
\|\theta\|_1^2 \leqslant C \mathcal{E}_0^2+C h^{2r-2} ,
\end{align}
where we have dropped the term involving $\theta_t$ on the left of the inequality and used
the equivalence of the $H^1$-norm and $H^1$-semi-norm for $\theta$.

It now remains to estimate $\mathcal{E}_0$ in \eqref{defeps1}. 
Again by Proposition \ref{prop1} with $k=1$, the last term in \eqref{defeps0} is bounded by $O(h^{2r-2})$.
Due to \eqref{eq:initial_approx} and \eqref{eq:icw}, we have
\begin{align*}
\left\| \nabla \theta(0) \right\| \leqslant \left\| w(0) - u_0 \right\|_{1} + \left\| u_0 - U_h(0) \right\|_{1}  \leqslant  C h^{r-1}\left\| u_0 \right\|_{r}
\end{align*}  
%
%
and similarly
\begin{align*}
\left\|  \theta_{t}(0) \right\| &= \left\|w_t(0)- U_{h,t}(0) \right\| =  \left\| w_t(0)- v_0 + v_0- U_{h,t}(0) \right\| \\
&\leqslant  \left\| \eta_{t}(0) \right\|  +  \left\|v_0 - v_{0,h} \right\| \\
&\leqslant Ch^{r},    
\end{align*}
which concludes the proof of the theorem. 
 \end{proof}

\section{Numerical experiments}
\label{sec:numexp}
Here we present several numerical experiments which validate the theory but also illustrate
the complex behavior of waves propagating across a time-varying medium. 
First, we consider two time-modulated smoothly varying media either with or without
added first-order gain/loss term. Then we consider the more realistic model of a one-dimensional
metamaterial as a chain of discontinuous time-modulated subwavelength resonators, which eventually leads
to localized strong amplification of the wave field. 

For the spatial discretization of the computational domain, $\Omega$, we use standard continuous $\mathcal{P}_2$ finite elements with mass-lumping \cite{CohenJolyRobertsTordjman01}. For the time discretization, we either use the standard
leapfrog method, or a combined leapfrog/Crank–Nicolson scheme with added first-order gain/loss term.
Given a sequence of increasingly finer meshes, we calculate the $L^2$- and $H^1$-errors against a
reference solution computed on a FE mesh 16 times finer than the finest mesh
used for the convergence test.

\subsection{Time-modulated  $\rho(x,t)$ and $\kappa(x,t)$}
We first consider the second-order wave equation in the interval $\Omega = (0,1)$ 
\[
\frac 1 {\kappa(x,t)} \, \frac{\partial^2 u}{\partial t^2}  - \frac{\partial}{\partial x} \left( \frac 1{\rho(x,t)} \, \frac{\partial u}{\partial x} \right) = 0,  \qquad x\in\Omega, \quad t>0,
\]
subject to homogeneous Dirichlet boundary conditions. The background medium is static with constant 
density $\rho_0 = 1$ and bulk modulus $\kappa_0 = 1$, yielding a unit wave speed $c_0 = 1$. 
To model a medium in which both inertial ($\rho$) and restorative ($\kappa$) properties 
oscillate within a  resonating region, we
introduce a smooth, localized and time-modulated perturbation centered about $x_r = 0.5$ as 
\begin{equation}
\label{eq:rhokappa}
\rho(x,t) = 1 + \alpha_\rho\, f(x) g(t), \qquad \kappa(x,t) = 1 + \alpha_\kappa\, f(x) g(t),
\end{equation}
with amplitudes $\alpha_\rho = 0.3$, $\alpha_\kappa = 0.5$, the smooth Gaussian profile
\[
f(x) = \frac{1}{2} \exp\left( -\frac{(x - x_r)^2}{2\sigma_r^2} \right), \quad \sigma_r = 0.2,
\]
in space, and the harmonic modulation $g(t) = \sin(2\pi t)$ in time. Hence both $\rho(x,t)$ and $\kappa(x,t)$ are strictly positive and smoothly varying.

The initial condition corresponds to a right-travelling Gaussian pulse initiated from the left of the resonating region about $x=x_0$,
\[
u(x, 0) = \exp\left( -\frac{(x - x_0)^2}{2\sigma^2} \right), \quad x_0 = 0.1,\ \sigma = 0.1,
\]
with the initial velocity set to initially ensure a purely rightward propagating wave inside the background medium.
Over time the wave propagates across the background with unit speed, enters the resonating region near $x\approx0.4$, interacts with the smoothly varying time-modulated medium, and exits the resonating region near $x\approx0.6$.

The semidiscrete finite-element formulation,
\[
M(t)\, \ddot u_h + K(t)\, u_h = 0,
\]
involves a time-dependent mass matrix, $M(t)$, and stiffness matrix, $K(t)$. For
time discretization, we opt for the standard second-order
leapfrog method:
\[
M(t^n)\left(\frac{u^{n+1} - 2u^n + u^{n-1}}{\Delta t^2}\right) + K(t^n)u^n = 0, \qquad n = 1, 2, \dots
\]
with $u^0$ and $u^1$ determined by the initial conditions.

In general both $M(t^n)$ and $K(t^n)$ need to be recomputed, that is, re-assembled and possibly inverted 
at every time-step. 
Due to the separable form of the material coefficients $\rho(x,t)$ and $\kappa(x,t)$ chosen here, however,
both can easily be computed as 
\[
M(t^n) = M_1 + \alpha_\rho\, g(t^n) M_2, \qquad
K(t^n) = K_1 + \alpha_\kappa\, g(t^n) K_2, 
\]
where the matrices $M_1$, $K_1$, $M_2$ and $K_2$ are time-independent and thus need only be computed once.
Here $M_1$ and $K_1$ correspond to the mass and stiffness matrices for the background medium whereas
$M_2$ and $K_2$ involve the spatial perturbation $f(x)$. The solution of the linear system for
$u^{n+1}$ nonetheless requires computing its Cholesky decomposition at every time-step, and not just once
as for a stationary medium. To avoid inverting $M(t^n)$ at every time-step, we apply standard 
mass-lumping to both $M_1$ and $M_2$, thus replacing them by diagonal approximations without affecting the convergence rate. 

   \begin{figure}[t]
    \centering
    \includegraphics[width=0.48\textwidth]{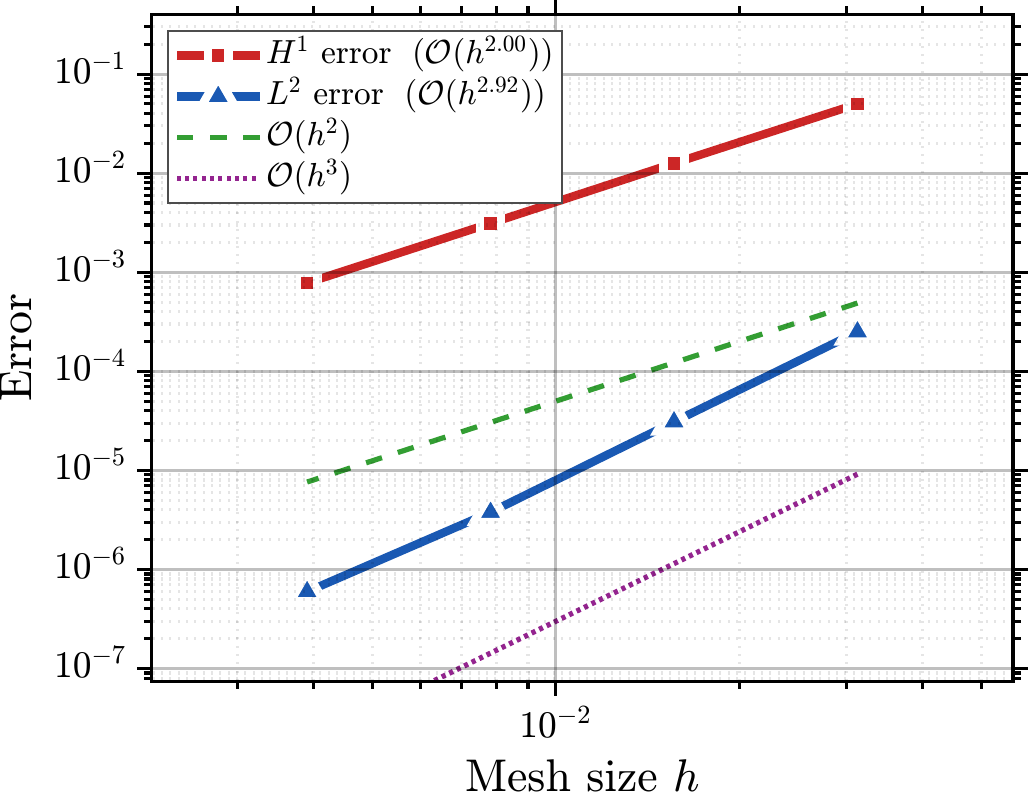}
    \caption{Time-modulated  $\rho(x,t)$ and $\kappa(x,t)$: $L^2$-error (blue) and $H^1$-error (red). 
    The dashed lines indicate the expected convergence rates $O(h^2)$ and $O(h^3)$ for comparison.}
    \label{fig:conv_scenario1}
\end{figure}

In Fig. \ref{fig:conv_scenario1} we observe that the FEM with leapfrog time-stepping achieves the expected
(optimal) second-order convergence using $\mathcal{P}_2$-FE and mass-lumping with respect to the $H^1$-norm. 
For the $L^2$-norm the method even achieves third-order convergence, as expected, as we let the
time-step $\Delta t$ here decrease proportionally to $h^{3/2}$ to match the spatial accuracy.

%
%
%
%

\subsection{Time-modulated $\rho(x,t)$, $\kappa(x,t)$ and $\sigma(x,t)$}
Next, we also include a gain/loss term which typically arises from temporal variations in the bulk modulus, a feature of certain time-modulated metamaterials. Hence we now consider the damped and driven second-order wave equation 
in $\Omega=(0,1)$,
\[
\frac 1 {\kappa(x,t)} \, \frac{\partial^2 u}{\partial t^2}  + \sigma(x,t) \, \frac{\partial u}{\partial t}  - \frac{\partial}{\partial x} \left( \frac 1{\rho(x,t)} \, \frac{\partial u}{\partial x} \right) = 0,  \qquad x\in\Omega, \quad t>0,
\]
with $\rho(x,t)$ and $\kappa(x,t)$ as in \eqref{eq:rhokappa} and 
\[
\sigma(x,t) = \beta_\sigma\, f(x) g(t), \qquad \beta_\sigma = 0.3\, .
\]
Due to the periodic sign change of the coefficient $\sigma(x,t)$, it either acts as a gain when negative (injecting energy) or as a loss when positive (removing energy), and thus actively couples the medium to the wave dynamics.


The semi-discrete FE formulation now yields the second-order system of ordinary differential equations,
\[
M(t) \ddot{u}_h + \Sigma(t) \dot{u}_h + K(t) u_h = 0,
\]
where the entries of the time-dependent loss/gain matrix $\Sigma(t)$ again correspond
to a mass-matrix yet with weight $\sigma(x,t)$.

%

After time discretization, the leapfrog/Crank–Nicolson update becomes
\[
M(t^n)\left(\frac{u^{n+1} - 2u^n + u^{n-1}}{\Delta t^2}\right)
+ \Sigma(t^n) \left(\frac{u^{n+1} - u^{n-1}}{2\Delta t}\right)
+ K(t^n)u^n = 0,\qquad n = 1, 2, \dots
\]
To avoid inverting $M(t^n) + (\Delta t/2)  \Sigma(t^n)$ at every time-step, we again apply standard 
mass-lumping techniques and take advantage of the separable form of the material coefficients. 

 \begin{figure}[htbp]
    \centering
    \includegraphics[width=0.48\textwidth]{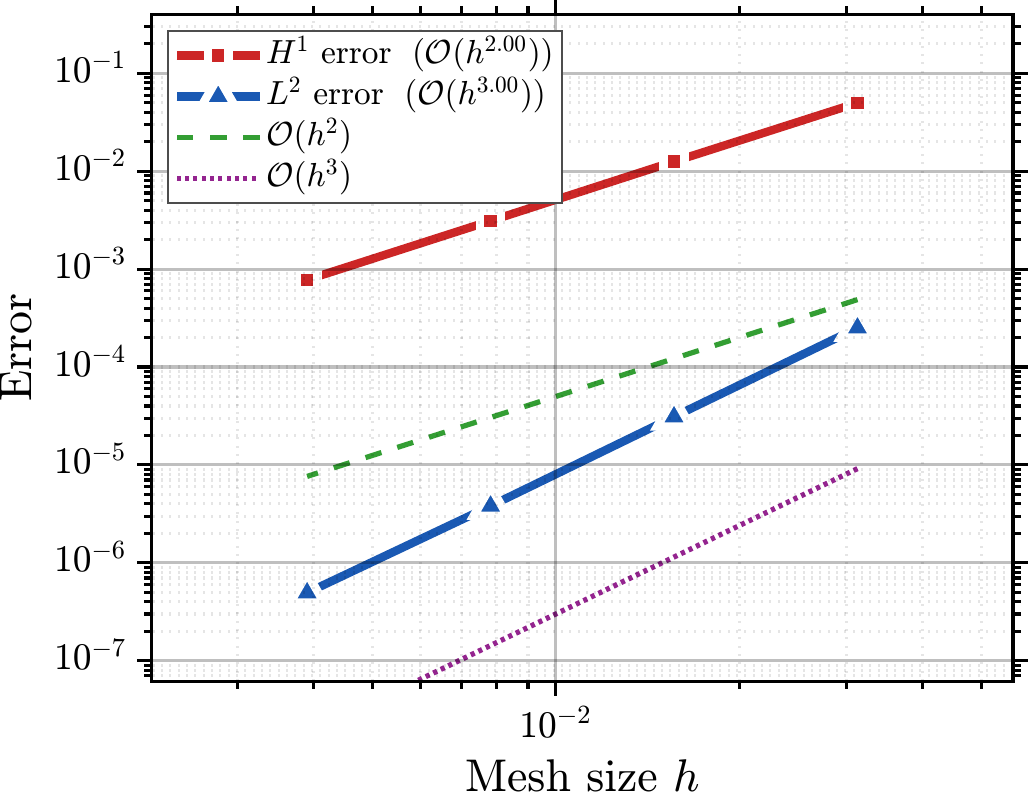}
    \caption{Time-modulated \(\rho(x,t)\), \(\kappa(x,t)\), and \(\sigma(x,t)\): $L^2$-error (blue) and $H^1$-error (red). 
    The dashed lines indicate the expected convergence rates $O(h^2)$ and $O(h^3)$ for comparison.
    }
    \label{fig:conv_scenario2}
\end{figure}

In Fig. \ref{fig:conv_scenario2}, we again observe that the FEM now with leapfrog/Crank-Nicolson time-stepping achieves the expected
(optimal) second-order convergence using $\mathcal{P}_2$-FE and mass-lumping with respect to the $H^1$-norm. 
For the $L^2$-norm, the method even achieves third-order convergence, as expected, as we again let the
time-step $\Delta t$ here decrease proportionally to $h^{3/2}$ to match the spatial accuracy.

%
 
\subsection{Chain of time-modulated subwavelength resonators}
To illustrate the complex behavior of waves propagating in a time-varying medium, we finally
consider waves travelling across a structured metamaterial operating in the subwavelength regime; hence, the characteristic scale of the microstructure is much smaller than the incident wavelength. 
Again we consider \eqref{eq:rhokappa}, but the medium now consists of $N_r = 50$ time-modulated resonators arranged in a uniform chain. The resonators $R_i = (x_i^-,x_i^+)$, each of unit length $x_i^+ - x_i^- = 1$, are separated by a gap $\ell_{ij} = 1$ of unit length, as shown in Fig. \ref{fig:resonators}, and thus form a periodic lattice in $x \in (0, 100)$. The resonating region is
embedded into a much larger computational domain $\Omega = (-500, 500)$ to
properly resolve the long-wavelength incident and scattered waves.

\begin{figure}[htbp]
\centering
\begin{tikzpicture}[>=stealth]
    \draw[blue, ultra thick, dotted] (0,0) -- (7,0);
    
    \def\resonatorX{{1, 2.7, 5}}
    
    \foreach \i [count=\j] in {1, 2, 3}{
        \draw[green!70!black, ultra thick] (\resonatorX[\j-1]-0.1,-0.1) rectangle (\resonatorX[\j-1]+0.6,0.1);
        
        \draw[fill, ultra thick] (\resonatorX[\j-1]-0.1,-0.15) rectangle (\resonatorX[\j-1]-0.08,-0.05);
        \node[below] at (\resonatorX[\j-1]-0.09,-0.15) {\footnotesize $x_{\i}^{-}$};
        
        \draw[fill, ultra thick] (\resonatorX[\j-1]+0.6,-0.15) rectangle (\resonatorX[\j-1]+0.62,-0.05);
        \node[below] at (\resonatorX[\j-1]+0.61,-0.15) {\footnotesize $x_{\i}^{+}$};
        
        \node[above] at (\resonatorX[\j-1]+0.25,0.15) {$R_{\i}$};
    }
    
    \draw[->] (7,0) -- (8,0) node[right] {$\mathbb{R}$};

    \draw[black, ultra thick, decorate, decoration={brace, amplitude=10pt}, yshift=3pt] (1.6,0.1) -- (2.6,0.1) node[midway, above=12pt, black] {$\ell_{12}$};
    \draw[black, ultra thick, decorate, decoration={brace, amplitude=10pt}, yshift=3pt] (3.3,0.1) -- (4.9,0.1) node[midway, above=12pt, black] {$\ell_{23}$};

    \draw[fill, ultra thick] (\resonatorX[0]+1,0.0) rectangle (\resonatorX[0]+1,-0.15);
    \node[below] at (\resonatorX[0]+1.02,-0.2) {\textbf{0}};    
    
\end{tikzpicture}
\caption{Schematic of three time-modulated resonators $R_i$ (green) embedded in a background medium (dashed blue line). The distance between consecutive resonators is denoted by $\ell_{i,i+1}$.}
\label{fig:resonators}
\end{figure}
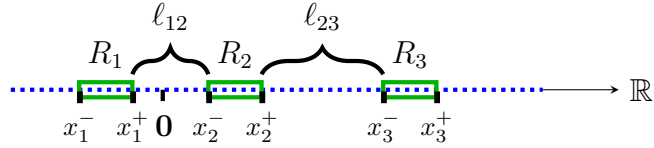

In contrast to the previous two numerical examples, the material parameters now exhibit a sharp contrast between the static background and the dynamically modulated resonators:
\begin{equation*}
\rho(x,t) = 
\begin{cases}
\rho_0 = 1, & x \in \Omega, \\[4pt]
\frac{\rho_r} {1 + \alpha_{\rho} \cos(\omega_{\rho} t) }, & x \in R_i,
\end{cases}
\qquad
\kappa(x,t) = 
\begin{cases}
\kappa_0 = 1, & x \in \Omega, \\[4pt]
\frac{\kappa_r} {1 + \alpha_{\kappa} \cos(\omega_{\kappa} t) }, & x \in R_i,
\end{cases}
\end{equation*}
where $\alpha_\rho = 0.2$, $\alpha_\kappa = 0.4$, $ \rho_r = \kappa_r =0.1$ and $\omega_\rho = \omega_\kappa=2\pi$.

\begin{figure}[t]
  \centering
  \begin{subfigure}[t]{0.45\textwidth}
    \centering
    \includegraphics[width=\textwidth]{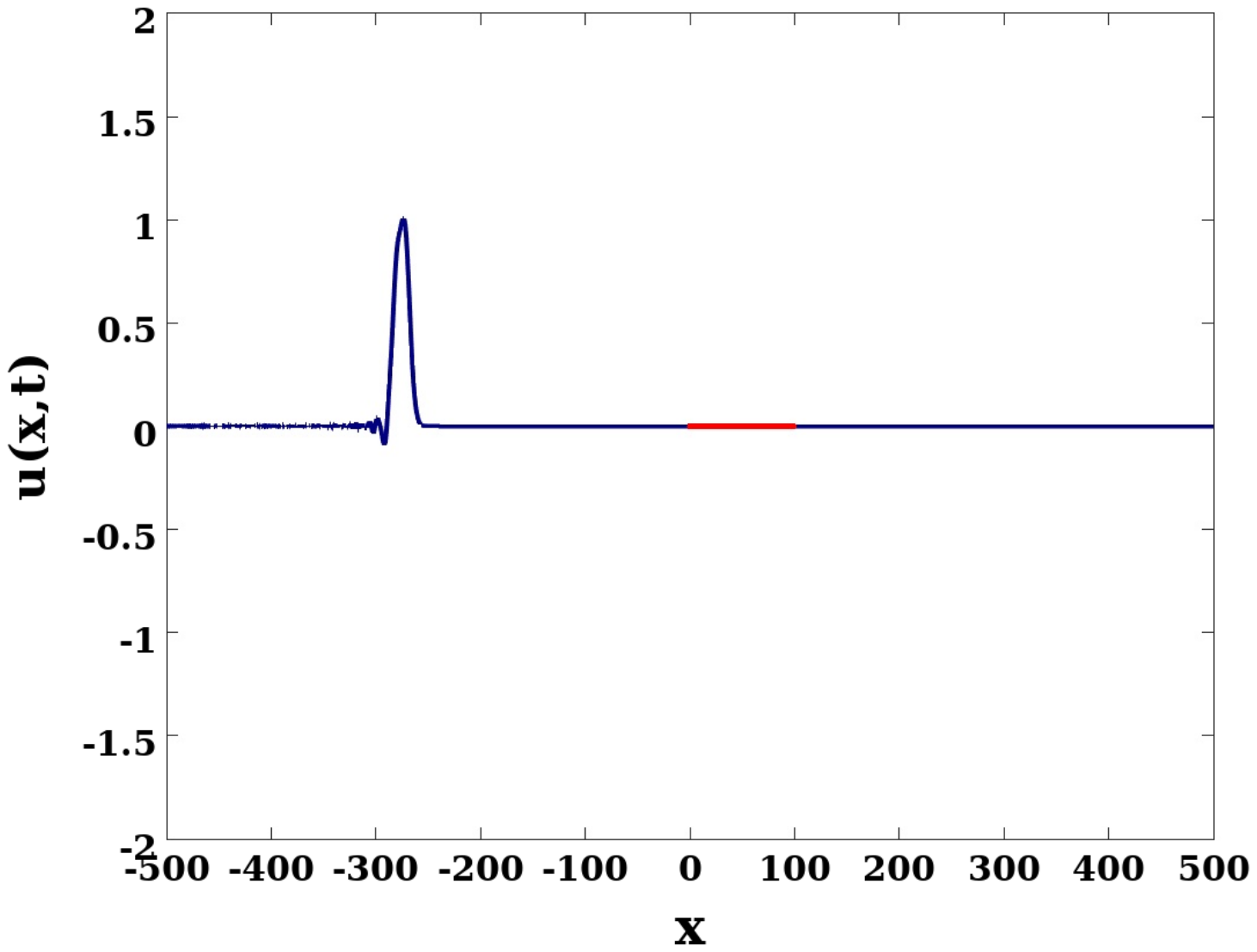}
  \end{subfigure}
  \hfill
  \begin{subfigure}[t]{0.45\textwidth}
    \centering
    \includegraphics[width=\textwidth]{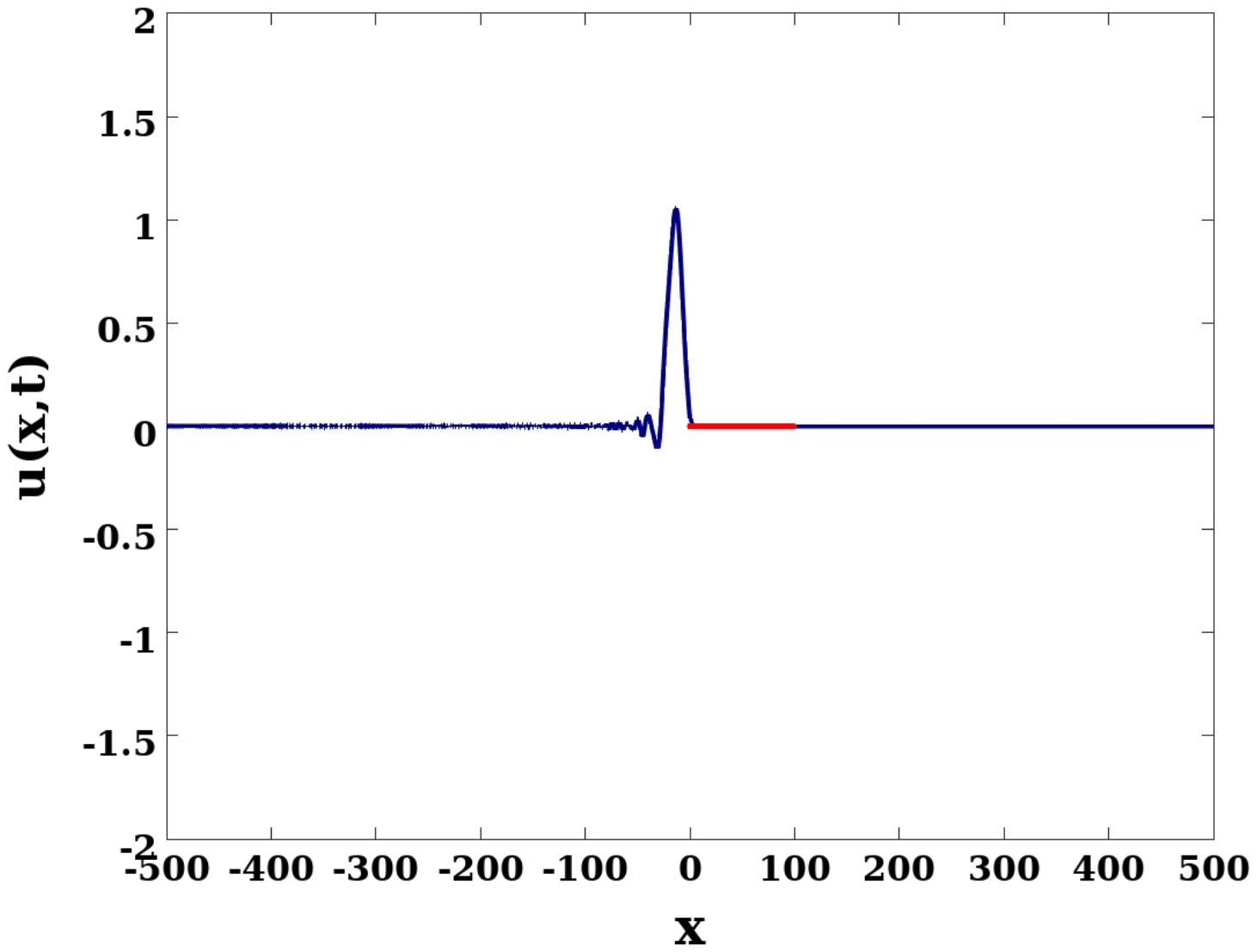}
  \end{subfigure}

  \vspace{0.3cm}

  \begin{subfigure}[t]{0.45\textwidth}
    \centering
    \includegraphics[width=\textwidth]{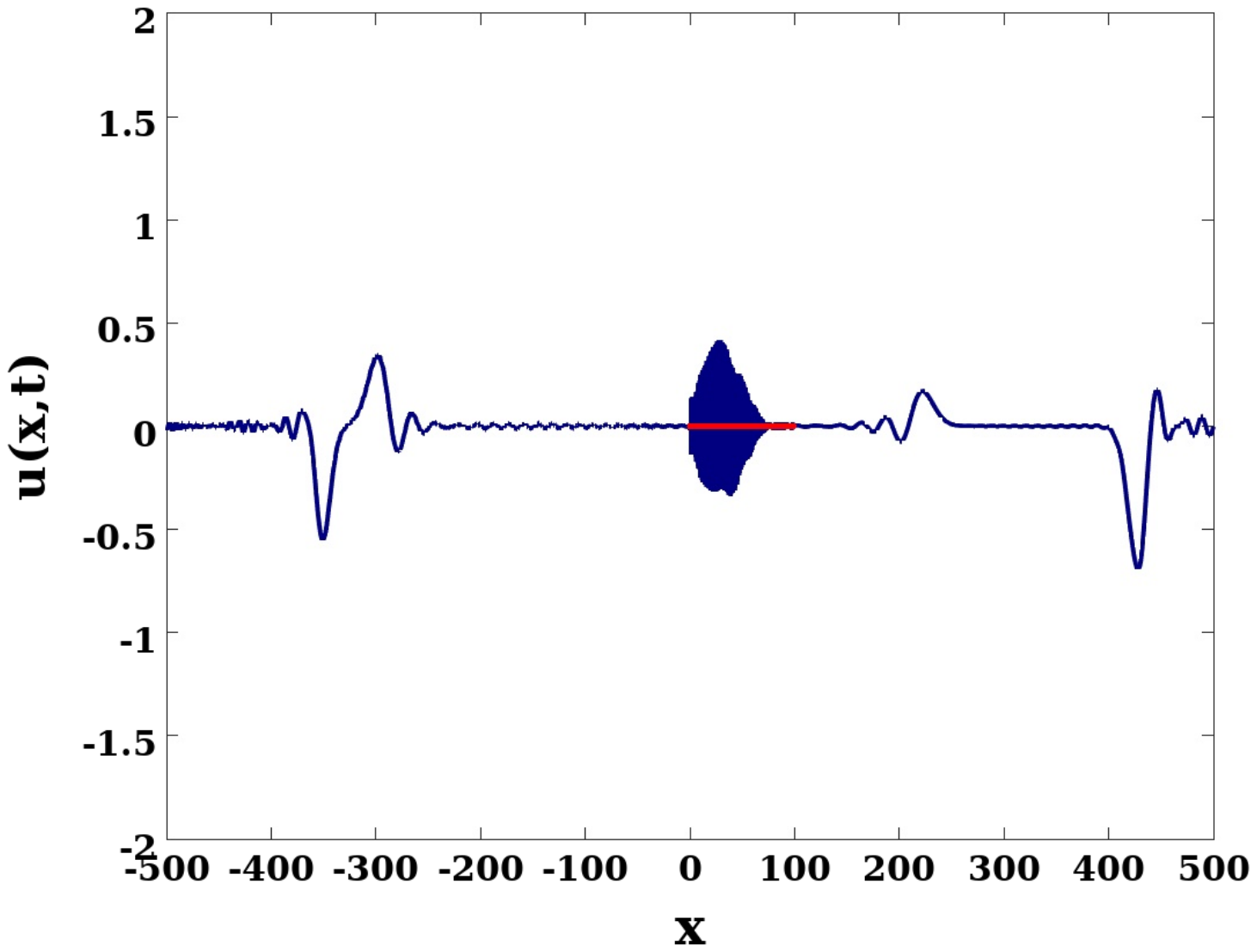}
  \end{subfigure}
  \hfill
  \begin{subfigure}[t]{0.45\textwidth}
    \centering
    \includegraphics[width=\textwidth]{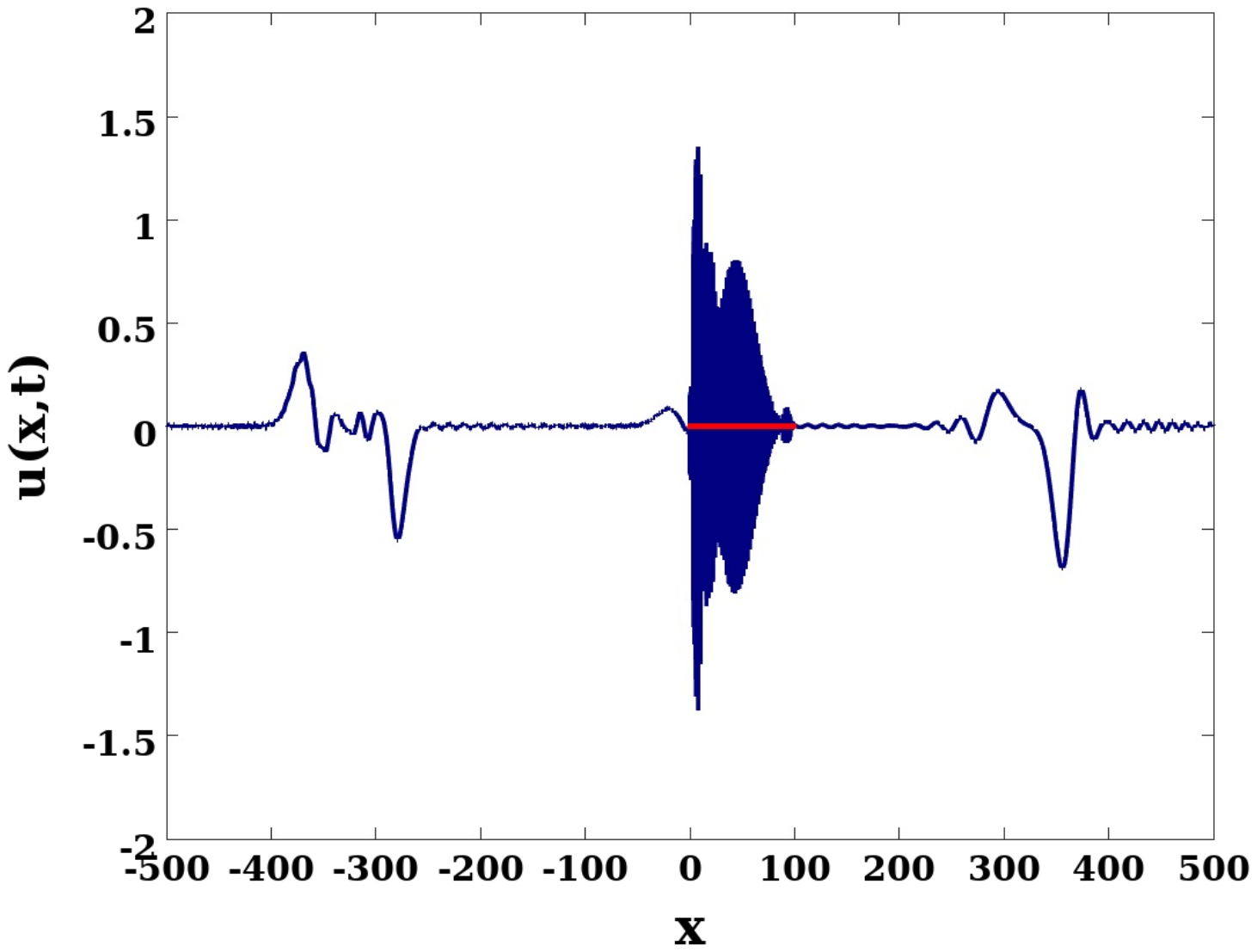}
  \end{subfigure}

  \caption{Chain of time-modulated subwavelength resonators: snapshots of the numerical solution at times $t = 238, 500, 1164, 1236$.}
  \label{fig:metamaterial_wave}
\end{figure}

In Fig.  \ref{fig:metamaterial_wave} we observe how a right-moving pulse enters the resonating time-modulated region
at about time $t \approx 800$. As the incident wave crosses and interacts with the time-modulated meta-material, 
%
%
we observe strong spatial localization and field amplification within the chain of time-modulated resonators, which can be interpreted as a manifestation of "wave localization in a space-time disordered medium". 
While the spatial periodicity induces Bragg scattering, the coherent temporal modulation introduces a dynamic scattering mechanism analogous to that described for purely time-varying random media \cite{PhysRevLett.127.094101}. 
In such media, the interplay of multiple scattering events in the time domain leads to exponential energy growth and log-normal field statistics.
  In our deterministic setup, the synchronized modulation acts as a coordinated series of parametric 'kicks,' transferring energy from the modulation pump to the wave and resulting in a hybrid localized mode
  which grows exponentially \cite{PhysRevLett.127.094101}.
%
Hence the drastic localized wave field enhancement is not mere amplification but evidence of a
macroscopically localized, exponentially growing state engineered by
spatiotemporal coherence.

 \section*{Acknowledgements} This work was supported by the Swiss National Science Foundation under grant SNF 200020-188583.
  
\bibliographystyle{abbrv}
\bibliography{references}

\end{document}